\documentclass[11pt,a4paper]{amsart}
\usepackage{latexsym}
\usepackage{prettyref}
\usepackage{float}
\usepackage{amsmath}
\usepackage[dvips]{graphicx}
\usepackage{graphics}
\usepackage{amssymb}
\usepackage{graphicx,psfrag}

\usepackage{hyperref}

\begin{document}

\newcommand{\RR}{\mathbb{R}^{2}}
\newcommand{\hh}[1]{\mathcal{H}_{\delta}^{#1}}
\newcommand{\HH}[1]{\mathcal{H}^{#1}}
\newcommand{\h}{\mathfrak{h}}
\newcommand{\e}{\varepsilon}
\newcommand{\N}{\mathbb{N}}
\newcommand{\R}{\mathbb{R}}
\newcommand{\RN}{\mathbb{R}^{n}}
\newcommand{\s}{\mathbb{S}}
\newcommand{\sub}{\subseteq}
\newcommand{\diam}{\text{diam}}

\newcommand{\D}{\displaystyle}

\renewcommand{\a}{\mathfrak{a}}
\renewcommand{\b}{\mathfrak{b}}

\date{}

\newtheorem{theorem}{Theorem}[section]
\newtheorem{lemma}[theorem]{Lemma}
\newtheorem{corollary}[theorem]{Corollary}
\newtheorem{proposition}[theorem]{Proposition}

\theoremstyle{definition}
\newtheorem{definition}[theorem]{Definition}
\newtheorem{example}[theorem]{Example}
\newtheorem{remark}[theorem]{Remark}

\title{Improving dimension estimates for Furstenberg-type sets}

\author{Ursula Molter and Ezequiel Rela}
\address{Departamento de
Matem\'atica \\ Facultad de Ciencias Exactas y Naturales\\ Universidad
de Buenos Aires\\ Ciudad Universitaria, Pabell\'on I\\ 1428 Capital
Federal\\ ARGENTINA\\ and CONICET, Argentina}
\email[Ursula Molter]{umolter@dm.uba.ar}
\email[Ezequiel Rela]{erela@dm.uba.ar}
\keywords{Furstenberg sets, Hausdorff dimension, dimension function.}
\subjclass{Primary 28A78, 28A80}

\thanks{This research  is partially supported by
Grants: PICT2006-00177 and UBACyT X149}



\begin{abstract}
In this paper we study the problem of estimating the generalized Hausdorff dimension of Furstenberg sets in the plane. For $\alpha\in(0,1]$, a set $F$ in the plane is said to be an $\alpha$-Furstenberg set if for each direction $e$ there is a line segment $\ell_e$ in the direction of $e$ for which $\dim_H(\ell_e\cap F)\ge\alpha$. It is well known that $\dim_H(F) \geq \max\{2\alpha, \alpha + \frac12\}$ - and it is also known that these sets can have zero measure at their critical dimension. By looking at general Hausdorff measures $\mathcal{H}^h$ defined for doubling functions, that need not be power laws, we obtain finer estimates for the size of the more general $h$-Furstenberg sets. Further, this approach allow us to sharpen the known bounds on the dimension of classical Furstenberg sets. 

The main difficulty we had to overcome, was that if $\mathcal H^{\h}(F) = 0$, there {\em always} exists $g \prec \h$ such that $\mathcal H^g(F) = 0$ (here $\prec$ refers to the natural ordering on general Hausdorff dimension functions). Hence, in order to estimate the measure of general Furstenberg sets, we have to consider dimension functions that are a true {\em step down} from the critical one. We provide rather precise estimates on the size of this step and by doing so, we can include a family of zero dimensional Furstenberg sets associated to dimension functions that grow faster than any power function at zero. With some additional growth conditions on these zero dimensional functions, we extend the known inequalities to include the endpoint $\alpha=0$.

\end{abstract}

\maketitle

\section{Introduction}

In this paper we study dimension properties of sets of Furstenberg type. We are able to sharpen the known bounds about the Hausdorff dimension of these sets using general doubling dimension functions for the estimates.

Let us recall the notion of \textit{Furstenberg sets}. 
For $\alpha$ in $(0,1]$, a subset $E$ of $\RR$ is called \textit{Furstenberg set} or $F_\alpha$-set if for each direction $e$ in the unit circle there is a line segment $\ell_e$ in the direction of $e$ such that the Hausdorff dimension of the set $E\cap\ell_e$ is equal or greater than $\alpha$.

 We will also say that such set $E$ belongs to the class $F_\alpha$. It is known (\cite{wol99b},  see also \cite{wol99a}, \cite{wol02}, \cite{wol03}, \cite{kt02a}, \cite{tao01} for related topics and \cite{KT01}, \cite{taofinite} for a discretized version of this problem) that for any $F_\alpha$-set $E\subseteq \RR$ the Hausdorff dimension ($\dim(E)$) must satisfy the inequality $\dim(E)\ge\max\{2\alpha,\alpha+\frac{1}{2}\}$
and there are examples of $F_\alpha$-sets $E$ with $\dim(E)\le \frac{1}{2}+\frac{3}{2}\alpha$. If we denote by
\begin{equation}
  \gamma(\alpha)=\inf \{\dim(E): E\in F_\alpha\},\nonumber
\end{equation}
then
\begin{equation}\label{eq:dim}
\max\{\alpha+\frac{1}{2} ; 2\alpha\}\le \gamma(\alpha)\le\frac{1}{2}+\frac{3}{2}\alpha,\qquad \alpha\in(0,1].
\end{equation}

In this paper we study a more general notion of Furstenberg sets. To that end we will use a finer notion of dimension already defined by Hausdorff \cite{hau19}. 

\begin{definition}
The following class of functions will be called \textit{dimension functions}. 
\begin{equation}
 \mathbb{H}:=\{h:[0,\infty)\to[0:\infty), 
 \text{non-decreasing, right continuous, } 
h(0)=0 \}.\nonumber
\end{equation} 
\end{definition}

The important subclass of those $h\in\mathbb{H}$ that satisfy a doubling condition will be denoted by $\mathbb{H}_d$:
\begin{equation}
	\mathbb{H}_d:=\left\{h\in\mathbb{H}: h(2x)\le C h(x) \text{ for some }C>0\right\}.\nonumber
\end{equation}

\begin{remark}
Clearly, if $h\in\mathbb{H}_d$, the same inequality will hold (with some other constant) if 2 is replaced by any other $\lambda>1$. We also remark that any concave function trivially belongs to $\mathbb{H}_d$.
\end{remark}

 As usual, the $h$-dimensional (outer) Hausdorff  measure $\HH{h}$ will be defined as follows. For a set $E\subseteq\RR$ and $\delta>0$, write
\begin{equation}
	\hh{h}(E)=\inf\left\{\sum_i h(\diam(E_i)):E\subset\bigcup_i^\infty E_i,  \diam(E_i)<\delta \right\}.\nonumber
\end{equation}
Then the $h$-dimensional Hausdorff measure $\HH{h}$ of $E$ is defined by
\begin{equation}
	\HH{h}(E)=\sup_{\delta>0 }\hh{h}(E).\nonumber
\end{equation} 
This notion generalizes the classical $\alpha$-Hausdorff measure to functions $h$ that are different to $x^\alpha$.
It is well known that a set of Hausdorff dimension $\alpha$ can have zero, positive or infinite $\alpha$-dimensional measure. The desirable situation, in general, is to work with a set which is \textit{truly} $\alpha$-dimensional, that is, it has positive and finite $\alpha$-dimensional measure. In this case we refer to this set as an $\alpha$-set.

Now, given an $\alpha$-dimensional set $E$ without this last property, one could expect to find in the class $\mathbb{H}$ an appropriate function $h$ to detect the precise ``size" of it. By that we mean that $0<\HH{h}(E)<\infty$, and in this case $E$ is referred to as an $h$-set.

We mention one example: A Kakeya set is a compact set containing a unit segment in every possible direction. It is known that there are Kakeya sets of zero measure and it is conjectured that they must have full Hausdorff dimension. The conjecture was proven by Davies \cite{dav71} in $\RR$ and remains open for higher dimensions. Since in the class of planar Kakeya sets there are several distinct types of two dimensional sets (i.e. with positive or null Lebesgue measure), one would like to associate a dimension function to the whole class. A dimension function $h\in\mathbb{H}$ will be called the exact Hausdorff dimension function of the class of sets $\mathcal{C}$ if 
\begin{itemize}
\item For every set $E$ in the class $\mathcal{C}$, $\mathcal{H}^h(E)>0$.
\item There are sets $E\in\mathcal{C}$ with $\mathcal{H}^h(E)<\infty$.
\end{itemize}

In the direction of finding the exact dimension of the class of Kakeya sets in $\RR$, Keich  \cite{kei99} has proven that in the case of the Minkowsky dimension 
the exact dimension function is $h(x)=x^2\log(\frac{1}{x})$. For the case of the Hausdorff dimension, he provided some partial results. Specifically, he shows that in this case the exact dimension function $h$ must decrease to zero at the origin faster than $x^2\log(\frac{1}{x})\log\log(\frac{1}{x})^{2+\e}$ for any given $\e>0$, but slower than $x^2\log(\frac{1}{x})$. This notion of speed of convergence to zero will allow us to define a partial order between dimension functions that extends the usual order on the power laws (see \prettyref{def:order}).

In this paper we are interested in the problem of studying the exact Hausdorff dimension of the class of Furstenberg-type sets (for the precise definition, see \prettyref{def:furs}). We are able to find lower bounds for the dimension function, i.e. for a given class of Furstenberg-type sets, we find a dimension function $h$ with the property that any set in the class has positive $\HH{h}$-measure.

For the construction of $h$-sets associated to certain sequences see the work of Cabrelli {\textit{et al}} \cite{CMMS04} (see also \cite{gms07}). We refer to the work of Olsen and Renfro \cite{ols06}, \cite{ols05}, \cite{ols03} for a detailed study of the exact Hausdorff dimension of the Liouville numbers $\mathbb{L}$, which is a known example of a zero dimensional set. Moreover, the authors prove that this is also a dimensionless set, i.e. there is no $h\in \mathbb{H}$ such that $0<\mathcal{H}^h(\mathbb{L})<\infty$ (equivalently, for any dimension function $h$, one has $\mathcal{H}^h(\mathbb{L})\in\{0,\infty\}$). In that direction, further improvements are due to Elekes and Keleti \cite{ek06}. There the authors prove much more than that there is no exact Hausdorff-dimension function for the set $\mathbb{L}$ of Liouville numbers: they prove that for any translation invariant Borel measure $\mathbb{L}$ is either of measure zero or has non-sigma-finite measure. So in particular they answer the more interesting question that there is no exact Hausdorff-dimension function for $\mathbb{L}$ even in the stronger sense when requiring only sigma-finiteness instead of finiteness.

If one only looks at the power functions, there is a natural total order given by the exponents. In $\mathbb{H}$ we also have a natural notion of order, but we can only obtain a \textit{partial} order.
\begin{definition}\label{def:order}
Let $g,h$ be two dimension functions. We will say that $g$ is dimensionally smaller than $h$ and write $g\prec h $ if and only if
	\begin{equation}
	\lim_{x\to 0^+}\dfrac{h(x)}{g(x)}=0.\nonumber
	\end{equation}	
\end{definition}

\begin{remark}
 Note that this partial order, restricted to the class of power functions, recovers the natural order mentioned above. That is,
\begin{equation}
	x^\alpha\prec x^\beta \iff \alpha<\beta.\nonumber
\end{equation} 
\end{remark}

Now we can make a precise statement of the problem. We begin with the definition of the Furstenberg-type sets.

\begin{definition}\label{def:furs}
Let $\h$ be a dimension function. A set $E\subseteq\RR$ is a Furstenberg set of type $\h$, or an $F_\h$-set, if for each direction $e\in\s$ there is a line segment $\ell_e$ in the direction of $e$ such that  $\HH{\h}(\ell_e \cap E)>0$. 	
\end{definition}
Note that this hypothesis is stronger than the one used to define the original Furstenberg-$\alpha$ sets. However, the hypothesis $\dim(E\cap\ell_e)\ge \alpha$ is equivalent to $\HH{\beta}(E\cap\ell_e)>0$ for any $\beta$ smaller than $\alpha$. If we use the wider class of dimension functions introduced above, the natural way to define $F_\h$-sets would be to replace the parameters $\beta<\alpha$ with two dimension functions satisfying the relation $h\prec \h$. But requiring  $E\cap\ell_e$ to have positive $\HH{h}$ measure for any $h\prec \h$ implies that it has also positive $\HH{\h}$ measure (Theorem 42, \cite{Rog70}).

Due to the existence of $F_\alpha$-sets with $\HH{\alpha}(E\cap\ell_e)=0$ for each $e$, it will be useful to introduce the following subclass of $F_\alpha$:

\begin{definition}
A set $E\subseteq\RR$ is an $F^+_\alpha$-set if for each $e\in\s$ there is a line segment $\ell_e$ such that  $\HH{\alpha}(\ell_e \cap E)>0$. 	
\end{definition}

\begin{remark}\label{rem:hypot}
	Given an $F_\h$-set $E$ for some $\h\in\mathbb{H}$, it is always possible to find two constants $m_E,\delta_E>0$ and a set $\Omega_E\sub\s$ of positive $\sigma$-measure such that
	\begin{equation}
			\hh{\h}(\ell_e\cap E)>m_E>0 \qquad \forall\delta<\delta_E\quad,\quad \forall e\in\Omega_E.\nonumber
	\end{equation} 
\end{remark}

For each $e\in\s$, there is a positive constant $m_e$ such that $\HH{\h}(\ell_e\cap E)>m_e$. Now consider the following pigeonholing argument. Let $\Lambda_n=\{e\in\s: \frac{1}{n+1}\le m_e<\frac{1}{n}\}$. At least one of the sets must have positive measure, since $\s=\cup_n \Lambda_n $. Let $\Lambda_{n_0}$ be such set and take $0<2m_E<\frac{1}{n_0+1}$. Hence
\begin{equation}
	\HH{\h}(\ell_e\cap E)>2m_E>0\nonumber
\end{equation} 
for all $e\in\Lambda_{n_0}$
Finally, again by pigeonholing, we can find $\Omega_E\subseteq\Lambda_{n_0}$ of positive measure and $\delta_E>0$ such 
\begin{equation}\label{eq:hypot}
	\hh{\h}(\ell_e\cap E)>m_E>0\qquad \forall e\in\Omega_E \quad \forall \delta<\delta_E.
\end{equation} 

To simplify notation throughout the paper, since inequality \prettyref{eq:hypot} holds for any Furstenberg set and we will only use the fact that $m_E$, $\delta_E$ and $\sigma(\Omega_E)$ are positive, it will be enough to consider the following definition of $F_\h$-sets:
\begin{definition}\label{def:general}
Let $\h$ be a dimension function. A set $E\subseteq\RR$ is Furstenberg set of type $\h$, or an $F_\h$-set, if for each $e\in\s$ there is a line segment $\ell_e$ in the direction of $e$ such that  $\hh{\h}(\ell_e \cap E)>1$ for all $\delta<\delta_E$ for some $\delta_E>0$.
\end{definition}

The purpose of this paper is to obtain an estimate of the dimension of an $F_\h$-set. By analogy to the classical estimate \prettyref{eq:dim}, we first note that if $\h$ is a general dimension function (not $x^\alpha$), $\alpha+\frac{1}{2}$ translates to $\h\sqrt{\cdot}$ and $2\alpha$ to $\h^2$. Hence, when aiming to obtain an estimate of the Hausdorff measure of our set $E$, the naive approach would be to prove that if a dimension function $h$ satisfies 
\begin{equation}\label{eq:prec01}
	 h\prec \h^2 \qquad \text{ or } \qquad h\prec \h\sqrt{\cdot},
\end{equation}
then $\HH{h}(E)>0$. However, there is no hope to obtain such a general result, since for the special case of the identity function $h(x)=x$, this requirement would contradict (again by Theorem 42, \cite{Rog70}) the existence of zero measure planar Kakeya sets. 

Therefore, it is clear that one needs to take a step down from the conjectured dimension function. The main result of this paper is to show that this step does not need to be as big as a power. It can be, for example, just the power of a $\log$. Precisely, we find conditions on the step that guarantee lower bounds on the dimension of $F_\h$-sets. Further, our techniques allow us to analyze Furstenberg-type sets of Hausdorff dimension zero. This can be done considering dimension functions $h$ that are smaller than $x^\alpha$ for any $\alpha>0$.

To keep the analogy with the classical Furstenberg sets, we will introduce the following notation:

\

\begin{definition}
Given two dimension functions $g,h\in\mathbb{H}$, we define the following quotients which are related to the step-size between two functions:

$\Delta_0(g,h)(x):=\Delta_0(x)=\frac{g(x)}{h(x)}$
\qquad	$\Delta_1(g,h)(x):=\Delta_1(x)=\frac{g(x)}{h^2(x)}$.
\end{definition}

When proving the first case of the inequalities in \eqref{eq:prec01}, the relevant quotient is $\Delta_1$, which gives the (better) bound $\dim \ge2\alpha$ in the classical case at the endpoint $\alpha=1$. At the other endpoint, $\alpha=0$, the best bound is $\dim\ge\alpha+\frac{1}{2}$ and the quotient to analyze here in our generalized problem is $\Delta_0$.

This paper is organized as follows. In \prettyref{sec:rem}, we introduce some further notation and prove a preliminary lemma to be used in the remainder of the paper. In  \prettyref{sec:h2} we prove the $\h^2$ bound, in \prettyref{sec:hsqrt+} the $\h\sqrt{\cdot}$ bound  under some positivity assumptions on the function $\h$ and in \prettyref{sec:F0} we drop this last condition to obtain a partial result on the zero dimensional Furstenberg sets. In addition we discuss our methods and study the Furstenberg problem in the extreme case of the counting measure. This is, roughly speaking, the case $\h\equiv 1$.

\section{Remarks, notation and more definitions}\label{sec:rem}

We will use the notation $A\lesssim B$ to indicate that there is a constant $C>0$ such that $A\le C B$, where the constant is independent of $A$ and $B$. By $A\sim B$ we mean that both $A\lesssim B$ and $B\lesssim A$ hold. On the circle $\s$ we consider the arclength measure $\sigma$. By $L^2(\s)$ we mean $L^2(\s,d\sigma)$. For each $e\in\s$, $\ell_e$ will be a unit line segment in the direction $e$. As usual, by a $\delta$-covering of a set $E$ we mean a covering of $E$ by sets $U_i$ with diameters not exceeding $\delta$.
We introduce the following notation:
\begin{definition}
Let $\mathfrak{b}=\{b_k\}_{k\in\N}$ be a decreasing sequence with $\lim b_k=0$. For any family of balls $\mathcal{B}=\{B_j\}$ with $B_j=B(x_j;r_j)$, $r_j\le 1$, and for any set $E$, we define
\begin{equation}\label{eq:jkb}
J^\b_k:=\{j\in\N:b_k< r_j\le b_{k-1}\},
\end{equation}
and 
\begin{equation}
E_k:=E\cap \bigcup_{j\in J^\mathfrak{b}_k}B_j.\nonumber
\end{equation}
In the particular case of the dyadic scale $\mathfrak{b}=\{2^{-k}\}$, we will omit the superscript and denote
\begin{equation}\label{eq:jk}
J_{k}:=\{j\in\N:2^{-k}<r_{j}\le2^{-k+1}\} 
\end{equation} 

\end{definition}

The next lemma introduces a technique we borrow from \cite{wol99b} to decompose  the set of all directions.

\begin{lemma}\label{lem:part}
	Let $E$ be an $F_\h$-set for some $\h\in\mathbb{H}$ and $\a=\{a_k\}_{k\in\N}\in\ell^1$ a non-negative sequence. Let $\mathcal{B}=\{B_j\}$ be a $\delta$-covering of $E$ with $\delta< \delta_E$ and let $E_k$ and $J_k$ be as above. Define
\begin{equation}
	\Omega_k:=\left\{e\in\s: \hh{\h}(\ell_e\cap E_k)\ge \frac{a_k}{2\|\a\|_1}\right\}.\nonumber
\end{equation}
Then $\s=\cup_k \Omega_k$.
\end{lemma}
\begin{proof}
Clearly $\Omega_k\subset \s$. To see why $\s=\cup_k \Omega_k$, assume that there is a direction $e\in\s$ that is not in any of the $\Omega_k$. Then for that direction we would have that
\begin{equation}
	1<\hh{\h}(\ell_e\cap E)\le\sum_k \hh{\h}(\ell_e\cap E\cap\bigcup_{j\in J_k}B_j)
	\le \sum_k\frac{1\ a_k}{2\|\a\|_1}=\frac{1}{2},\nonumber
\end{equation}
which is a contradiction.
\end{proof}

As a final remark we note that in the following sections our aim will be to prove essentially
\begin{equation}
	\sum_j h(r_j)\gtrsim 1,
\end{equation}  
provided that $h$ is a small enough dimension function. The idea will be to use the dyadic partition of the covering to obtain that
\begin{equation}
	\sum_j h(r_j)\gtrsim \sum_{k=0}^\infty h(2^{-k})\#J_k.\nonumber
\end{equation}  
The lower bounds we need will be obtained if we can prove lower bounds on the quantity $J_k$ in terms of the function $h$ but independent of the covering.

\section{
The \texorpdfstring{$\h\to \h^2$}{hah2} bound }\label{sec:h2}

In this section we generalize the first inequality of \prettyref{eq:dim}, that is, $\dim(E)\ge 2\alpha$ for any $F_\alpha$-set. For this, given a dimension function $h\prec \h^2$, we impose some sufficient growth conditions on the gap $\Delta_1(x):=\frac{h(x)}{\h^2(x)}$ to ensure that $\HH{h}(E)>0$. We have the following theorem:

\begin{theorem}\label{thm:htoh2}
	Let $\h\in\mathbb{H}_d$ be a dimension function and let $E$ be an $F_\h$-set. Let $h\in\mathbb{H}$  such that $h\prec \h^2$. If ${\D\sum_k} \h(2^{-k})\sqrt{\frac{k}{h(2^{-k})}}<\infty$, then $\HH{h}(E)>0$.
\end{theorem}

The main tool for the proof of this theorem will be an $L^2$ bound for the Kakeya maximal function on $\RR$. 

For an integrable function on $\RN$, the Kakeya maximal function at  scale $\delta$ will be $f^*_\delta:\s^{n-1}\to \R$,
\begin{equation}
	f^*_\delta(e)=\sup_{x\in\RN}\frac{1}{|T_e^\delta(x)|}\int_{T_e^\delta(x)}|f(x)|\ dx\qquad e\in\s^{n-1},\nonumber
\end{equation}
where $T_e^\delta(x)$ is a $1\times \delta$-tube (by this we mean a tube of length 1 and cross section of radius $\delta$) centered at $x$ in the direction $e$. It is well known that in $\RR$ the Kakeya maximal function satisfies the bound (see \cite{wol99b})

\begin{equation}\label{eq:maximal}
		\big\|f_\delta^*\big\|^2_2\lesssim \log(\frac{1}{\delta})\|f\|_2^2.
\end{equation} 		

It is also known that the $\log$ growth is necessary (see \cite{kei99}), because of the existence of Kakeya sets of zero measure in $\RR$. See also \cite{mit02} for estimates on the Kakeya maximal function with more general measures on the circle.

We now prove \prettyref{thm:htoh2}. We remark that since this theorem says, roughly speaking, that the dimension of an $F_\h$-set should be about $\h^2$, the step down must be taken from this dimension function. This is the role played by $\Delta_1(h,\h^2)(x)=\frac{h(x)}{\h^2(x)}$ in this section.

\begin{proof}
By \prettyref{def:general}, since $E\in F_\h$, we have
\begin{equation}
	\hh{\h}(\ell_e\cap E)>1\nonumber
\end{equation} 
for all $e\in\s$ and for any $\delta<\delta_E$. 

Let $\{B_j\}_{j\in\N}$ be a covering of $E$ by balls with $B_j=B(x_j;r_j)$. We need to bound $\sum_j h(2r_j)$ from below. Since $h$ is non-decreasing, it suffices to obtain the  bound 

\begin{equation}\label{eq:sum}
	\sum_j h(r_j)\gtrsim 1
\end{equation}  
for any $h\in \mathbb{H}$ satisfying the hypothesis of the theorem. Clearly we can restrict ourselves to $\delta$-coverings with $\delta<\frac{\delta_E}{5}$.

Define $\a=\{a_k\}$ with $a_k=\sqrt{\frac{k}{\Delta_1(2^{-k})}}$. Also define, as in the previous section, for each $k\in\N$, $J_k=\{j\in \N: 2^{-k}< r_j\le 2^{-k+1}\}$ and $E_k=E\cap\D\cup_{j\in J_k}B_j$.
Since by hypothesis $\a\in\ell^1$, we can apply \prettyref{lem:part} to obtain the decomposition $\s=\bigcup_k \Omega_k$ associated to this choice of $\a$.

We will apply the maximal function inequality to a weighted union of indicator functions. For each $k$, let $F_k=\D\bigcup_{j\in J_k}B_j$ and define the function 
\[
f:=\h(2^{-k})2^k\D\chi_{F_k}.
\]

We will use the $L^2$ norm estimates for the maximal function.
The $L^2$ norm of $f$ can be easily estimated as follows:
	\begin{eqnarray*}
    	\|f\|_{2}^{2} & = & \h^2(2^{-k})2^{2k}\int_{\cup_{J_k}B_{j}}dx\\
                  & \lesssim & \h^2(2^{-k})2^{2k}\sum_{j\in J_{k}}r_j^2 \\
		& \lesssim & \h^2(2^{-k}) \#J_k,
	\end{eqnarray*}
since $r_j\le 2^{-k+1}$ for $j\in J_k$.
Hence, 
\begin{equation}\label{eq:L2above}
	\|f\|_{2}^{2}\lesssim \#J_k \h^2(2^{-k}).
\end{equation} 
Now fix $k$ and consider the Kakeya maximal function $f^*_\delta$ of level $\delta=2^{-k+1}$ associated to the function $f$ defined for this value of $k$.

In $\Omega_k$ we have the following pointwise lower estimate for the maximal function. Let $\ell_e$ be the line segment such that $\hh{\h}(\ell_e\cap E)>1$, and let $T_e$ be the rectangle of width $2^{-k+2}$ around this segment. 
Define, for each $e\in\Omega_k$, 
\begin{equation}
 J_k(e):=\{j\in J_k: \ell_e\cap E\cap B_j\neq\emptyset\}.
\end{equation}

With the aid of the Vitali covering lemma, we can select a subset of disjoint balls $\widetilde{J}_k(e)\sub J_k(e)$ such that 
\begin{equation}
	\bigcup_{j\in J_k(e)}B_j \sub \bigcup_{j\in\widetilde{J}_k(e)}B(x_j;5r_j).\nonumber
\end{equation} 

Note that every ball $B_j$, $j\in J_k(e)$, intersects $\ell_e$ and therefore at least half of $B_j$ is contained in the rectangle $T_e$, yielding $|T_e\cap B_j|\ge\frac{1}{2}\pi r_j^2$. Hence, by definition of the maximal function, using that $r_j\ge 2^{-k+1}$ for $j\in J_k(e)$,
\begin{eqnarray*}
|f^{*}_{2^{-k+1}}(e)| 	& \geq 	& \frac{1}{|T_e|}\int_{T_e}f\ dx =\frac{\h(2^{-k})2^{k}}{|T_{e}|}\left|T_e\cap \cup_{J_k(e)}B_{j}\right|\\
	& \gtrsim &\h(2^{-k})2^{2k} \left|T_e\cap \cup_{\widetilde{J}_k(e)}B_{j}\right|\\
	& \gtrsim & \h(2^{-k})2^{2k}\sum_{j\in \widetilde{J}_{k}(e)}r^2_j\\
	&\gtrsim &\h(2^{-k})\#\widetilde{J}_{k}(e)\\
	& \gtrsim & \sum_{\widetilde{J}_{k}(e)}\h(r_j).
    \end{eqnarray*}
Now, since
\begin{equation}
	\ell_e\cap E_k\sub \bigcup_{j\in J_k(e)}B_j\sub \bigcup_{j\in \widetilde{J}_k(e)}B(x_j;5r_j)
\end{equation}
and for $e\in\Omega_k$ we have $\hh{\h}(\ell_e\cap E_k)\gtrsim a_k$, we obtain 
\begin{equation}
	|f^{*}_{2^{-k+1}}(e)|
	\gtrsim\sum_{\widetilde{J}_{k}(e)}\h(r_j)
	\gtrsim\sum_{j\in \widetilde{J}_{k}(e)}\h(5r_j)\gtrsim a_k.\nonumber
\end{equation} 

Therefore we have the estimate
\begin{equation}\label{eq:maximalbelow}
	\|f^{*}_{2^{-k+1}}\|_2^2\gtrsim\int_{\Omega_k} |f^{*}_{2^{-k+1}}(e)|^2\ d\sigma \gtrsim a_k^2\ \sigma(\Omega_k)=\frac{\sigma(\Omega_k)k}{\Delta_1(2^{-k})}.
\end{equation} 
Combining \prettyref{eq:L2above},  \prettyref{eq:maximalbelow} and using the maximal inequality \prettyref{eq:maximal}, we obtain
\begin{equation}
	\frac{\sigma(\Omega_k)k}{\Delta_1(2^{-k})}\lesssim \|f^{*}_{2^{-k+1}}\|_2^2\lesssim \log(2^k)\|f\|_2^2\lesssim k\#J_k \h^2(2^{-k}),\nonumber
\end{equation} 
and therefore
\begin{equation}
	\frac{\sigma(\Omega_k)}{h(2^{-k})}\lesssim \#J_k.\nonumber
   \end{equation}

Now we are able to estimate the sum in \prettyref{eq:sum}. Let $h$ be a dimension function satisfying the hypothesis of \prettyref{thm:htoh2}. We have
\begin{eqnarray*}
	\sum_j h(r_j)	& \ge & \sum_k h(2^{-k})\#J_k\\
			& \gtrsim & \sum_k \sigma(\Omega_k) \ge\sigma(\s)>0.
\end{eqnarray*}
\end{proof}

Applying this theorem to the class $F^+_\alpha$, we obtain a sharper lower bound on the generalized Hausdorff dimension:

\begin{corollary}\label{cor:coro1}
	Let $E$ an $F^+_\alpha$-set. If $h$ is any dimension function satisfying $h(x)\ge C x^{2\alpha} \log^{1+\theta}(\frac{1}{x})$ for $\theta>2$ then $\HH{h}(E)>0$.
\end{corollary}

\begin{remark}
	At the endpoint $\alpha=1$, this estimate is worse than the one due to Keich. He obtained, using strongly the full dimension of a ball in $\RR$, that if $E$ is an $F^+_1$-set and $h$ is a dimension function satisfying the bound $h(x)\ge C x^2 \log(\frac{1}{x})\left(\log\log(\frac{1}{x})\right)^{\theta}$  for  $\theta>2$, then $\HH{h}(E)>0$.
\end{remark}

\begin{remark}
Note that the proof above relies essentially on the $L^1$ and  $L^2$ size of the ball in $\RR$, not on the dimension function $\h$. Moreover, we only use the ``gap" between $h$ and $\h^2$ (measured by the function $\Delta_1$). This last observation leads to conjecture that this proof can not be used to prove that an $F_\h$-set has positive $\h^2$ measure, since in the case of  $\h(x)=x$, as we remarked in the introduction, this would contradict the existence of Kakeya sets of zero measure in $\RR$.

Also note that the absence of conditions on the function $\h$ allows us to consider the ``zero dimensional'' Furstenberg problem. However, this bound does not provide any substantial improvement, since the zero dimensionality property of the function $\h$ is shared by the function $\h^2$. This is because the proof above, in the case of the $F_\alpha$-sets, gives the worse bound ($\dim(E)\ge 2\alpha$) when the parameter $\alpha$ is in $(0,\frac{1}{2})$. 
\end{remark}

\section{The  \texorpdfstring{$\h\to \h\sqrt{\cdot}$}{hsqrt+} bound
}\label{sec:hsqrt+}

In this section we will turn our attention to those functions $h$ that satisfy the bound $h(x)\lesssim x^\alpha$ for $\alpha\le\frac{1}{2}$. For these functions we are able to improve on the previously obtained bounds. We need to impose some growth conditions on the dimension function $\h$. This conditions can be thought  of as imposing a lower bound on the dimensionality of $\h$ to keep it away from the zero dimensional case. 

\begin{remark} 
\ Throughout this section, the expected dimension function should be about $\h\sqrt{\cdot}$. We therefore need a step down from this function. For this, we will look at the gap $\Delta_0(x)=\frac{h(x)}{\h(x)}$.
\end{remark}

The next lemma says that we can split the $\h$-dimensional mass of a set $E$ contained in an interval $I$ into two sets that are positively separated.

\begin{lemma}\label{lem:split}
 Let $\h\in\mathbb{H}$, $\delta>0$, $I$ an interval and $E\sub I$. Let $\eta>0$ be such that $\h^{-1}(\frac{\eta}{8})<\delta$ and $\hh{\h}(E)\ge\eta>0$. Then there exist two subintervals $I^-$, $I^+$ that are $\h^{-1}(\frac{\eta}{8})$-separated and  with $\hh{\h}(I^{\pm}\cap E)\gtrsim \eta$.
\end{lemma}

\begin{proof}
Let $t=\h^{-1}(\frac{\eta}{8})$ and subdivide $I$ in $N$ ($N\ge3$) consecutive (by that we mean that they intersect only at endpoints and leave no gaps between them) subintervals $I_j$ such that $|I_j|=t$ for $1\le j\le N-1$ and $|I_N|\le t$. Since $|I_j|<\delta$ and $\h(|I_j|)\le\frac{\eta}{8}$, we have
\begin{equation}\label{eq:inter}
	\hh{\h}(E\cap I_j)\le \h(|I_j|)\le\frac{\eta}{8}
\end{equation} 
and
\begin{equation}
	\eta \le \hh{\h}(E)=\hh{\h}\left(\bigcup_j E\cap I_j\right)\le\sum_j\hh{\h}(E\cap I_j).\nonumber
\end{equation} 

Now we can group the subintervals in the following way. Let $n$ be the first index for which we have $\sum_{j=1}^n \hh{\h}(E\cap I_j)>\frac{\eta}{4}$.

Since $\sum_{j=1}^{n-1} \hh{\h}(E\cap I_j)\le\frac{\eta}{4}$, and by \prettyref{eq:inter} the mass of each interval is not too large, we have the bound 
\begin{equation}
	\frac{\eta}{4}<\sum_{j=1}^n \hh{\h}(E\cap I_j)\le(\frac{1}{4}+\frac{1}{8})\eta=\frac{3\eta}{8}.\nonumber
\end{equation} 
Take $I^-=I_1\cup\dots\cup I_n$, skip the interval $I_{n+1}$, and consider $I^+$ to be the union of the remaining intervals. It is easy to see that
\begin{equation}
	\sum_{j=1}^{n+1} \hh{\h}(E\cap I_j)\le\frac{\eta}{2},\nonumber
\end{equation} 
and therefore
\begin{equation}
	\sum_{j=n+2}^{N} \hh{\h}(E\cap I_j)\ge\frac{\eta}{2}.\nonumber
\end{equation} 
So, we obtain $\hh{\h}(I^{\pm}\cap E)\ge \frac{\eta}{4}$ and the intervals $I^-$ and $I^+$ are $|I_j|$-separated. But $|I_j|=\h^{-1}(\frac{\eta}{8})$, so the lemma is proved.
\end{proof}

The next lemma will provide estimates for the number of lines with certain separation property that intersect two balls of a given size.

\begin{lemma}\label{lem:conobolas}
Let $\mathfrak{b}=\{b_k\}_{k\in\N}$ be a decreasing sequence with $\lim b_k=0$. Given a family of balls $\mathcal{B}=\{B(x_j;r_j)\}$, we define $J^\b_k$ as in \prettyref{eq:jkb} and let $\{e_i\}_{i=1}^{M_k}$ be a $b_k$-separated set of directions. Assume that for each $i$ there are two line segments $I_{e_i}^+$ and $I_{e_i}^-$ lying on a line in the direction $e_i$  that are $s_k$-separated for some given $s_k$ 
Define $\Pi_k=J^\b_{k}\times
J^\b_{k}\times\{1,..,M_k\}$ and $\mathcal{L}^\b_k$ by

\begin{equation}
\mathcal{L}^\b_k:=\left\{(j_{+},j_{-},i)\in \Pi_k:
    	I_{e_{i}}^-\cap B_{j_-}\neq\emptyset\ 
	I_{e_{i}}^+\cap B_{j_+}\neq\emptyset
 \right\}.\nonumber
\end{equation} 
If $\frac{1}{5}s_k>b_{k-1}$ for all $k$, then
\begin{equation}
\#\mathcal{L}^\b_k\lesssim\frac{b_{k-1}}{b_k}\frac{1}{s_k}\left(\#J^\b_{k}\right)^{2}.\nonumber
\end{equation} 
\end{lemma}

\begin{proof}
Consider a fixed pair $j_-,j_+$ and its associated $B_{j_-}$ and $B_{j_+}$ 
We will use as distance between two balls the distance between the centers, and for simplicity we denote $d(j_-,j_+)=d(B_{j_-},B_{j_+})$. If $d(j_-,j_+)<\frac{3}{5}s_k$ then there is no $i$ such that $(j_-,j_+,i)$ belongs to $\mathcal{L}^\b_k$.

Now, for  $d(j_-,j_+)\ge\frac{3}{5}s_k$, we will look at the special configuration given by \prettyref{fig:conobolas} when we have $r_{j_-}=r_{j_+}=b_{k-1}$ and the balls are tangent to the ends of $I^-$ and $I^+$. This will give a bound for any possible configuration, since in any other situation the cone of allowable directions is narrower.

\begin{figure}[ht]
       \begin{center}
	\psfrag{bj-}{$B_{j_-}$}
	\psfrag{bj+}{$B_{j_+}$}
	\psfrag{bk-1}{$b_{k-1}$}
	\psfrag{sk}{$s_k$}
	\psfrag{Imenos}{$I^-$}
	\psfrag{Imas}{$I^+$}
			\includegraphics[scale=.4]{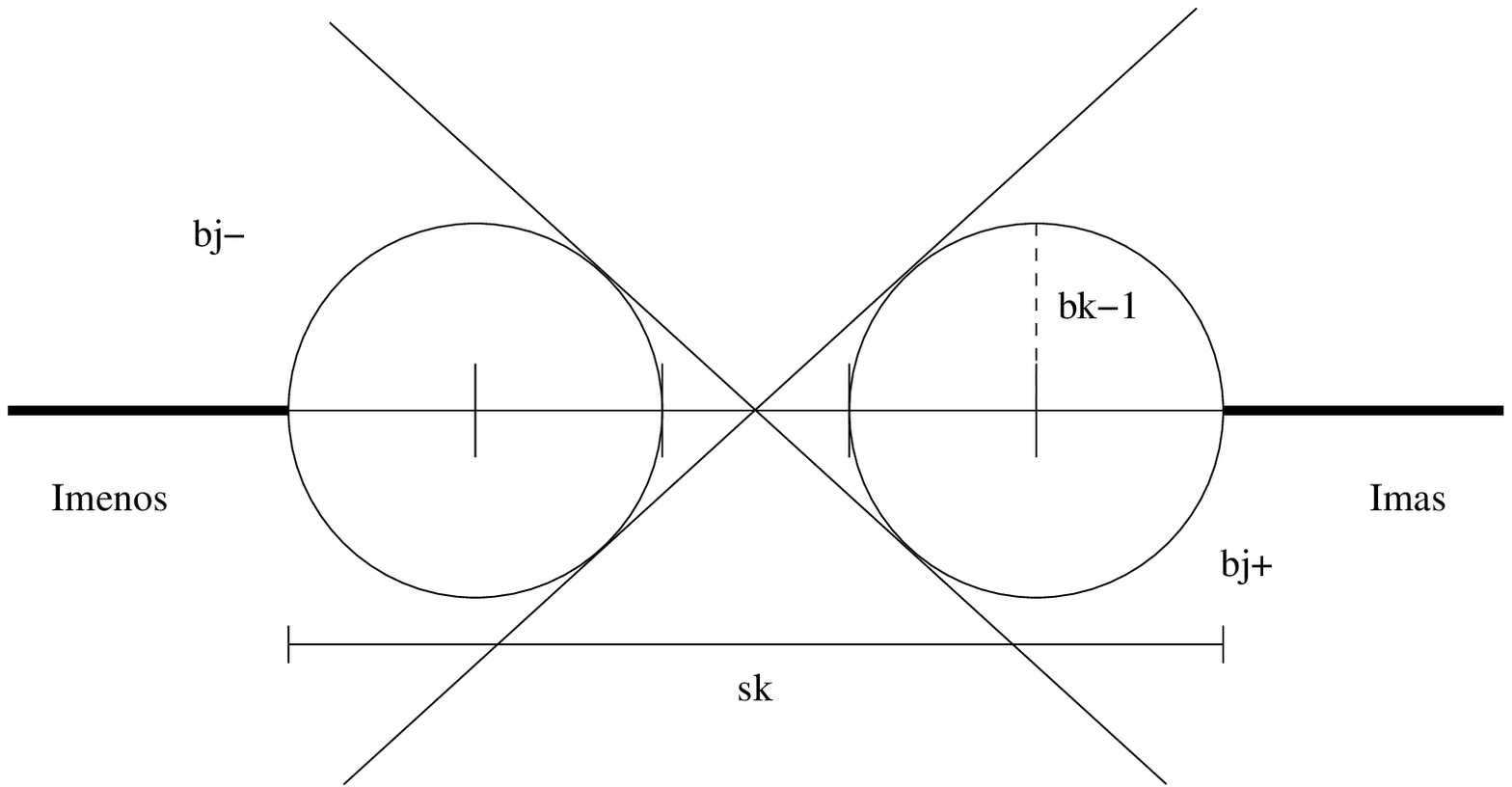}
		\caption{}
	\label{fig:conobolas}
	\end{center}
\end{figure}

Let us focus on one half of the cone (\prettyref{fig:cono}). Let $\theta$ be the width of the cone. In this case, we have to look at  $\frac{\theta}{b_k}$ directions that are $b_k$-separated. Further, we note that $\theta=\frac{2\theta_k}{s_k}$, where $\theta_k$ is the bold arc at distance $s_k/2$ from the center of the cone.

\begin{figure}[ht]
       \begin{center}
		\psfrag{arco1}{$\theta$}
		\psfrag{uno}{$1$}
		\psfrag{arcos_k/2}{$\theta_k$}
		\psfrag{b_k-1}{$b_{k-1}$}
		\psfrag{s_k/2}{$\frac{s_k}{2}$}
		\includegraphics[scale=.3]{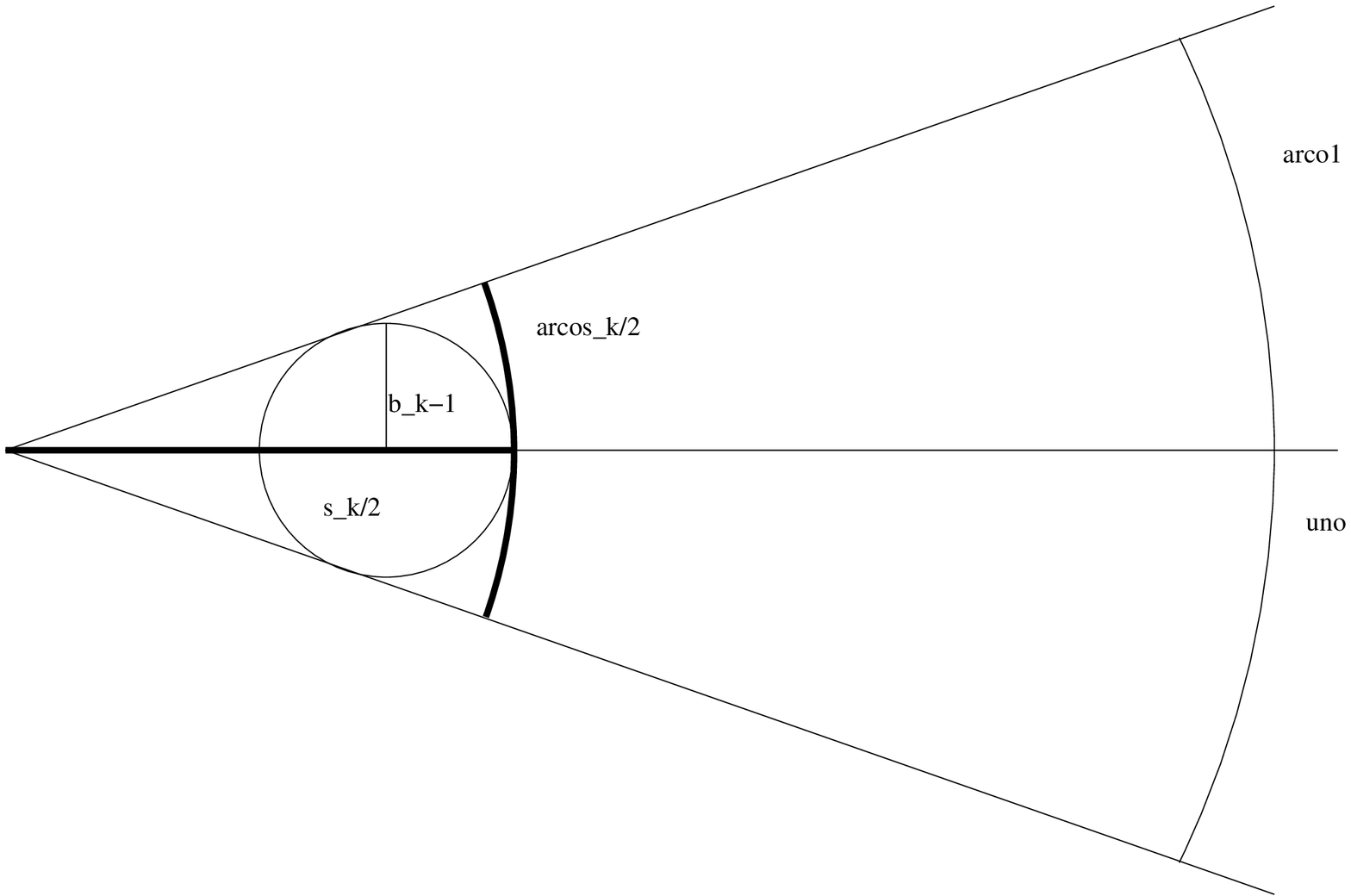}
		\caption{}\label{fig:cono}
	\end{center}
\end{figure}

It is clear that $\theta_k\sim b_{k-1}$, and therefore the number $D$ of lines in $b_k$-separated directions with non-empty intersection with $B_{j_-}$ and $B_{j_+}$ has to satisfy $D\le \frac{\theta}{b_k}=\frac{2\theta_k}{s_kb_k}\sim\frac{b_{k-1}}{b_k}\frac{1}{s_k}$.

The lemma follows by summing on all pairs $(j_-,j_+)$.
\end{proof}

Now we can prove the main result of this section. We have the following theorem:

\begin{theorem}\label{thm:htohsqrt}
	Let $\h\in\mathbb{H}_d$ be a dimension function such that $\h(x)\lesssim x^\alpha$ for some $0<\alpha<1$ and $E$ be an $F_\h$-set. Let $h\in\mathbb{H}$ with $h\prec \h$.  If ${\D\sum_k} \left(\frac{\h(2^{-k})}{h(2^{-k})}\right)^{\frac{2\alpha}{2\alpha+1}}<\infty$, then $\HH{h\sqrt{\cdot}}(E)>0$.
\end{theorem}

\begin{proof}

We begin in the same way as in the previous section. Again by \prettyref{def:general}, since $E\in F_\h$, we have $\hh{\h}(\ell_e\cap E)>1$
for all $e\in\s$ for any $\delta<\delta_E$. 

Consider the sequence $\a=\left\{\Delta_0^{-\frac{2\alpha}{2\alpha+1}}(2^{-k})\right\}_{k}$. Let $k_0$ be such that
\begin{equation}\label{eq:hiplemasplit}
\h^{-1}\left(\dfrac{a_k}{16\|\a\|_1}\right)<\delta_E \qquad \text{for any }k\ge k_0.
\end{equation} 
Now take any $\delta$-covering $\mathcal{B}=\{B_j\}$ of $E$ by balls with $\delta<\min\{\delta_E, 2^{-k_0}\}$. Using \prettyref{lem:part} we obtain $\s=\bigcup_k\Omega_k$ with 

\begin{equation}\label{eq:mass}
	\Omega_k=\left\{e\in\Omega: \hh{\h}(\ell_e\cap E_k)\ge \frac{ a_k}{2\|\a\|_1}\right\}.
\end{equation}
Again we have  $E_k=E\cap\bigcup_{j\in J_k}B_j$, but by our choice of $\delta$, the sets $E_k$ are empty for $k<k_0$. Therefore the same holds trivially for $\Omega_k$ and we have that  $\s=\bigcup_{k\ge k_0}\Omega_k$.

The following argument is Remark 1.5 in \cite{wol99b}. Since for each $e\in\Omega_k$ we have  \prettyref{eq:hiplemasplit}, we can apply \prettyref{lem:split} with $\eta=\frac{a_k}{2\|\a\|_1}$ to $\ell_e\cap E_k$. Therefore we obtain two intervals $I_e^-$ and $I_e^+$, contained in $\ell_e$ with
\begin{equation}
	\hh{\h}(I^{\pm}_e\cap E_k)\gtrsim a_k\nonumber
\end{equation} 
that are $\h^{-1}(ra_k)$-separated for $r=\frac{1}{16\|\a\|_1}$.

Let $\{ e^k_{j}\}_{j=1}^{M_k}$ be a $2^{-k}$-separated subset of $\Omega_k$. Therefore $M_k\gtrsim2^{k}\sigma(\Omega_k)$. Define $\Pi_k:=J_{k}\times
J_{k}\times\{1,..,M_k\}$ and 
\begin{equation}\label{eq:tauk}
\mathcal{T}_k:=\left\{(j_{-},j_{+},i)\in \Pi_k:
    	I_{e_{i}}^-\cap E_{k}\cap B_{j_-}\neq\emptyset\ 
	I_{e_{i}}^+\cap E_{k}\cap B_{j_+}\neq\emptyset
 \right\}.
\end{equation} 
We will count the elements of $\mathcal{T}_k$ in two different ways. First, fix $j_{-}$ and $j_{+}$ and count for how many values of $i$ the triplet $(j_{-},j_{+},i)$ belongs to $\mathcal{T}_k$. 

For this, we will apply \prettyref{lem:conobolas} for the choice $\b=\{2^{-k}\}$. The estimate we obtain is the number of $2^{-k}$-separated directions $e_i$, that intersect simultaneously the balls $B_{j_-}$ and $B_{j_+}$, given that these balls are separated. We obtain
\begin{equation}\label{eq:jfijo}
\#\mathcal{T}_k\lesssim\frac{1}{\h^{-1}(ra_k)}\left(\#J_{k}\right)^{2}.
\end{equation} 
Second, fix $i$. In this case, we have by hypothesis that  $\hh{\h}(I_{e_i}^{+}\cap E_{k})\gtrsim a_k$, so $\sum_{j_{+}}\h(r_{j_+})\gtrsim a_k$. Therefore, 

\begin{equation}
 a_k\lesssim\sum_{(j_-,j_+,i)\in\mathcal{T}_k}\h(r_{j_+})\le K \h(2^{-k}),\nonumber
\end{equation} 

where $K$ is the number of elements of the sum. Therefore $K\gtrsim \frac{a_k}{\h(2^{-k})}$. The same holds for $j_{-}$, so
\begin{equation}\label{eq:ifijo}
\#\mathcal{T}_k\gtrsim M_k\left(\frac{a_k}{\h(2^{-k})}\right)^2.
\end{equation} 
Combining the two bounds,
\begin{eqnarray*}
\#J_{k} & \gtrsim &
(\#\mathcal{T}_k)^{1/2}\h^{-1}(ra_k)^{1/2}\\
 & \gtrsim & M_k^{1/2}\frac{a_k}{\h(2^{-k})}\h^{-1}(ra_k)^{1/2}\\
 & \gtrsim &2^{\frac{k}{2}}\sigma(\Omega_k)^{1/2}\frac{a_k}{\h(2^{-k})}\h^{-1}(ra_k)^{1/2}.
\end{eqnarray*}
Consider now a dimension function $h\prec \h$ as in the hypothesis of the theorem. Then again
\begin{eqnarray}\label{eq:cotahaushsqrt}
\sum_j h(r_j)r_j^{1/2}& \ge &\sum_{k} \h(2^{-k})\Delta_0(2^{-k})2^{-\frac{k}{2}}\#J_k\\\nonumber
&\gtrsim&\sum_{k\ge k_0} \sigma(\Omega_k)^{1/2}\Delta_0(2^{-k})a_k\h^{-1}(ra_k)^{1/2}.
\end{eqnarray}

To bound this last expression, we use first that there exists $\alpha\in(0,1)$ with $\h(x)\lesssim x^\alpha$ and therefore $\h^{-1}(x)\gtrsim x^{\frac{1}{\alpha}}$. We then recall the definition of the sequence $\a$,  $a_k=\Delta_0(2^{-k})^{-\frac{2\alpha}{1+2\alpha}}$ to obtain
 
\begin{eqnarray}\label{eq:dompower}
\sum_j h(r_j)r_j^{1/2}& \gtrsim &\sum_{k\ge k_0} \sigma(\Omega_k)^{1/2}\Delta_0(2^{-k})a_k^{\frac{1+2\alpha}{2\alpha}}\\
&=& \sum_{k\ge k_0}\sigma(\Omega_k)^{1/2}\gtrsim 1\nonumber.
\end{eqnarray}

\end{proof}

The next corollary follows from \prettyref{thm:htohsqrt} in the same way as \prettyref{cor:coro1} follows from \prettyref{thm:htoh2}.

\begin{corollary}\label{cor:coro2}
	Let $E$ be an $F^+_\alpha$-set. If $h$ is a dimension function satisfying $h(x)\ge C x^\alpha\sqrt{x} \log^{\theta}(\frac{1}{x})$ for $\theta>\frac{1+2\alpha}{2\alpha}$ then $\HH{h}(E)>0$.

\end{corollary}

\begin{remark}
Note that at the critical value $\alpha=\frac{1}{2}$, we can compare \prettyref{cor:coro1} and \prettyref{cor:coro2}. The first says that in order to obtain $\HH{h}(E)>0$ for an $F_{\frac{1}{2}}^+$-set $E$ 
it is sufficient to require that the dimension function $h$ satisfies the bound $h(x)\ge Cx\log^\theta(\frac{1}{x})$ for $\theta>3$. On the other hand, the latter says that it is sufficient that $h$ satisfies the bound $h(x)\ge x\log^\theta(\frac{1}{x})$ for $\theta>2$. In both cases we prove that an $F_{\frac{1}{2}}^+$-set must have Hausdorff dimension at least 1, but \prettyref{cor:coro2} gives a better estimate on the logarithmic gap.
\end{remark}

\section{\texorpdfstring{$F_0$}{F0} - sets
}\label{sec:F0}

In this section we look at a class of very small Furstenberg sets. We will study, roughly speaking, the extremal case of $F_0$-sets and ask ourselves if inequality \prettyref{eq:dim} can be extended to this class. According to the definition of $F_\alpha$-sets, this class should be the one formed by sets having a zero dimensional linear set in every direction. We will call a dimension function $\h$ ``zero dimensional'' if $\h\prec x^\alpha$ for all $\alpha>0$.
Let us introduce the following subclasses of $F_0$:
\begin{table}[H]
\centering
\begin{tabular}{rl}
  $F_0^k$: & $E\in F_0^k$ if it contains at least $k$ points in every direction.\\ 
&\\
$F_0^\N$: & $E\in F_0^\N$ if it contains at least countable points in every direction.\\
&\\
$F_0^\h$: & $E\in F_0^\h$ if it belongs to $F_\h$ for a zero dimensional $\h\in\mathbb{H}$.
\end{tabular}
\end{table}
Our approach to the problem, using dimension functions, allows us to tackle the problem about the dimensionality of these sets in some cases.
We study the case of $F_\h$-sets associated to one particular choice of $\h$. We will look at the function $\h(x)=\dfrac{1}{\log(\frac{1}{x})}$ as a model of ``zero dimensional" dimension function. Our next theorem will show that in this case inequality \prettyref{eq:dim} can indeed be extended.
The trick here will be to replace the dyadic scale on the radii in $J_k$ with a faster decreasing sequence $\b=\{b_k\}_{k\in\N}$. 

The main difference will be in the estimate of the quantity of lines in $b_k$-separated directions that intersect two balls of level $J_k$ with a fixed distance $s_k$ between them. 
This estimate is given by \prettyref{lem:conobolas}.

We can prove the next theorem, which provide a class of examples of zero dimensional $F_\h$-sets.

\begin{theorem}
	Let $\h(x)=\frac{1}{\log(\frac{1}{x})}$ and let $E$ be an $F_\h$-set. Then $\dim(E)\ge \frac{1}{2}$.
\end{theorem}

\begin{proof}
Take a non-negative sequence $\b$ which will be determined later. We will apply the splitting \prettyref{lem:split} as in the previous section. For this, take $k_0$ as in \prettyref{eq:hiplemasplit} associated to the sequence $\a=\{{k^{-2}}\}_{k\in\N}$. Now, for a given generic $\delta$-covering of $E$ with $\delta<\min\{\delta_E,2^{-k_0}\}$, we use  \prettyref{lem:part} to obtain a decomposition $\s=\bigcup_{k\ge k_0}\Omega_k$ with
\begin{equation}
	\Omega_k=\left\{e\in\s: \hh{\h}(\ell_e\cap E_k)\ge ck^{-2}\right\},\nonumber
\end{equation}
where $E_k=E\cap \bigcup_{J_k^\b}B_j$, $J_k^\b$ is the partition of the radii  associated to $\b$ and $c>0$ is a suitable constant. 

We apply the splitting \prettyref{lem:split} to $\ell_e\cap E_k$  to obtain two $\h^{-1}(c k^{-2})$-separated intervals $I^-_e$ and $I^+_e$ with $\HH{\h}_\delta(I^\pm_e\cap E_k)\gtrsim k^{-2}$.

Now, let $\{e^k_j\}_{j=1}^{M_k}$ be a $b_k$-separated subset of $\Omega_k$. Therefore $M_k\gtrsim \Omega_k/b_k$.

We also define, as in \prettyref{thm:htohsqrt}, $\Pi_k:=J^\b_{k}\times
J^\b_{k}\times\{1,..,M_k\}$ and 
\begin{equation}
\mathcal{T}^\b_k:=\left\{(j_{-},j_{+},i)\in \Pi_k:
    	I_{e_{i}}^-\cap E_{k}\cap B_{j_-}\neq\emptyset\ 
	I_{e_{i}}^+\cap E_{k}\cap B_{j_+}\neq\emptyset
 \right\}.\nonumber
\end{equation}

By \prettyref{lem:conobolas}, we obtain
\begin{equation}\label{eq:otraescala}
\#\mathcal{T}^\b_k\lesssim\frac{b_{k-1}}{b_k}\frac{1}{\h^{-1}(ck^{-2})}(\#J^\b_{k})^{2},
\end{equation} 
and the same calculations as in \prettyref{thm:htohsqrt} (inequality \prettyref{eq:ifijo}) yield
\begin{equation}
\#J^\b_{k} \gtrsim \left(\frac{\sigma(\Omega_k)}{b_{k-1}}\right)^{1/2}\frac{\h^{-1}(ck^{-2})^{1/2}}{k^2\h(b_{k-1})}\ge
\left(\frac{\sigma(\Omega_k)}{b_{k-1}}\right)^{1/2}\frac{e^{-ck^2}}{k^2}.\nonumber
\end{equation}

Now we estimate a sum like \prettyref{eq:cotahaushsqrt}. For  $\beta<\frac{1}{2}$ we have
\begin{eqnarray}\label{eq:sumafinal}\nonumber
\sum_j r_j^{\beta}& \ge &\sum_k b_k^{\beta}\#J_k\\\nonumber
	&\ge&\sum_k \sigma(\Omega_k)^{1/2}
	\frac{b_k^\beta}{b_{k-1}^{\frac{1}{2}}}\frac{e^{-ck^2}}{k^2}\\
	&\gtrsim & \sqrt{\sum_k \sigma(\Omega_k) \frac{b_k^{2\beta}}{b_{k-1}}\frac{1}{e^{ck^2}k^4}}.
\end{eqnarray}
In the last inequality we use that the terms are all non-negative. The goal now is to take some rapidly decreasing sequence such that the factor $\frac{b_k^{2\beta}}{b_{k-1}}$ beats the factor $k^{-4}e^{-ck^2}$. 

Let us take $0<\e<\frac{1-2\beta}{2\beta}$  and consider the hyperdyadic scale $b_k=2^{-(1+\e)^k}$. With this choice, we have
\begin{equation}
	\frac{b_k^{2\beta}}{b_{k-1}}=2^{(1+\e)^{k-1}-(1+\e)^{k}2\beta}=2^{(1+\e)^{k}(\frac{1}{1+\e}-2\beta)}.\nonumber
\end{equation} 
Replacing this in inequality \prettyref{eq:sumafinal} we obtain
\begin{eqnarray*}
\left(\sum_j r_j^{\beta}\right)^2& \ge &\sum_k \sigma(\Omega_k)
			2^{(1+\e)^k(\frac{1}{1+\e}-2\beta)} 
		\frac{e^{-ck^2}}{k^4} \\
& \ge & \sum_k \sigma(\Omega_k) \frac{2^{(1+\e)^k(\frac{1}{1+\e}-2\beta)}}{e^{ck^2}k^4}.
\end{eqnarray*}
Finally, since by the positivity of $\frac{1}{1+\e}-2\beta$ the double exponential in the numerator grows much faster than the denominator, we obtain 
\begin{equation}
	\frac{2^{(1+\e)^k(\frac{1}{1+\e}-2\beta)}}{e^{ck^2}k^4}\gtrsim 1,
\end{equation} 
and therefore $ \left(\sum_j r_j^{\beta}\right)^2\gtrsim\sum_k \sigma(\Omega_k)\gtrsim 1$

\end{proof}

\begin{corollary}
	Let $\theta>0$. If $E$ is an  $F_\h$-set with $\h(x)=\frac{1}{\log^\theta(\frac{1}{x})}$ then $\dim(E)\ge \frac{1}{2}$.
\end{corollary}
\begin{proof}
	This follows immediately, since in this case the only change will be  $\h^{-1}(ck^{-2})=\frac{1}{e^{(ck^2)^{\frac{1}{\theta}}}},$
so the double exponential still grows faster and therefore $\frac{2^{(1+\e)^k(\frac{1}{1+\e}-2\beta)}}{e^{(ck^2)^{\frac{1}{\theta}}}k^4}\gtrsim 1.$
\end{proof}

This shows that there is a whole class of $F_0$-sets that must be at least $\frac{1}{2}$-dimensional.

Now, in the opposite direction, we will show some examples of very small $F_0$-sets. The first observation is that it is possible to construct $F_0^k$-sets and even $F_0^\N$-sets with Hausdorff dimension not exceeding $\frac{1}{2}$. This can be done with some suitable modifications of the construction made in \cite{wol99b}, Remark 1.5, (p. 10). There, for each $0 < \alpha \leq 1$,  an $F_\alpha$-set is constructed whose dimension is not greater than $\frac{1}{2}+\frac{3}{2}\alpha$. It is straightforward to modify that construction for it to hold even at the endpoint $\alpha=0$. 

We also include the following example of an $F_0^2$-set $G$ of dimension zero. It will be constructed using the next result, which is Example 7.8 (p. 104) in \cite{fal03}.
In that example, Falconer constructs sets $E, F\sub [0,1]$ with $\dim(E)=\dim(F)=0$ and such that $[0,1]\sub E+F$.

\begin{figure}[h]
       \begin{center}
	\psfrag{E}{$E\times \{1\}$}
	\psfrag{F}{$-F\times \{0\}$}
	\psfrag{x}{$x$}
	\psfrag{-y}{$-y$}
	\psfrag{tita}{$\theta$}
	\includegraphics[scale=.58]{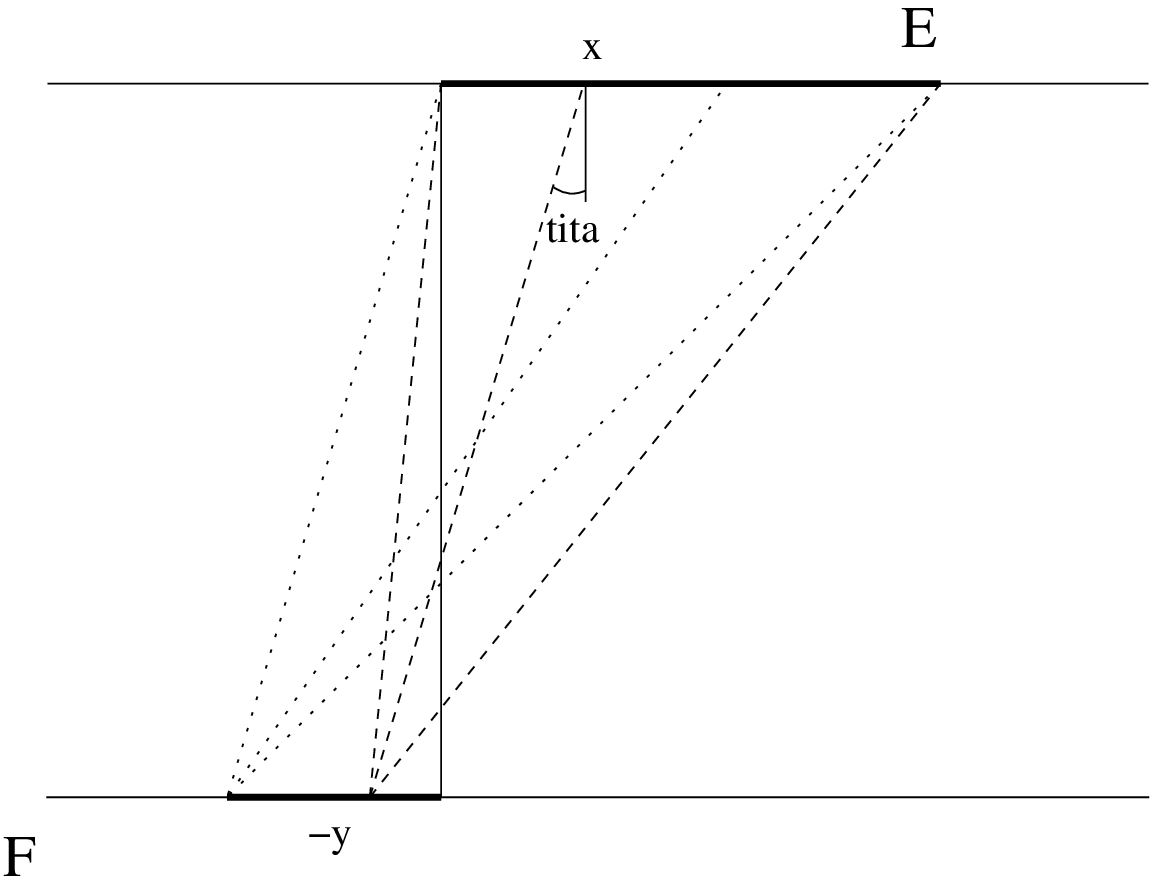}
	\caption{}\label{fig:f02}
	\end{center}
\end{figure}

Consider $G=E\times \{1\}\cup -F\times \{0\}$. This set $G$ has clearly dimension $0$, and contains two points in every direction $\theta\in [0;\frac{\pi}{4}]$. For, if $\theta\in [0;\frac{\pi}{4}]$, let $c=\tan(\theta)$, so $c\in[0,1]$. By the choice of $E$ and  $F$, we can find $x\in E$ and $y\in F$ with $c=x+y$.

The points $(-y,0)$ and $(x,1)$ belong to $G$ and determine a segment in the direction $\theta$ (\prettyref{fig:f02}).

\section{Acknowledgments}

We would like to thank to Michael T. Lacey for fruitful conversations during his visit to the Department of Mathematics at the University of Buenos Aires.

We also thank the anonymous referee for extremely careful reading of the manuscript and pointing out many subtle improvements which made this paper more readable.




\end{document}